\documentclass[11pt]{article}
\usepackage[margin=1.1in]{geometry}
\usepackage{float}
\usepackage{hyperref}
\usepackage{titling}
\usepackage{blindtext}
\usepackage{amsmath}
\usepackage{amssymb}
\usepackage{mathtools}
\usepackage[numbers]{natbib}
\usepackage{tikz}
\usepackage{graphicx}
\usepackage{siunitx}
\usepackage{gensymb}
\usepackage{url, verbatim}
\usepackage{mathrsfs}
\usepackage{amsthm}
\usepackage{cleveref}
\usepackage{tkz-euclide}
\usepackage{titlesec}
\usepackage{url}
\usepackage{enumerate}
\usepackage{pgfplots}
\usetikzlibrary{calc,patterns,quotes,shapes,arrows,through,intersections,decorations.pathreplacing}
\graphicspath{ {./KleinianImages/} }

\newtheorem{thm}{Theorem}[section]
\newtheorem{cor}[thm]{Corollary}
\newtheorem{lem}[thm]{Lemma}

\newtheorem{ques}[thm]{Question}
\numberwithin{equation}{section}

\date{}

\newcommand{\ubox}{\overline{\text{dim}}_\text{B}}
\newcommand{\lbox}{\underline{\text{dim}}_\text{B}}
\newcommand{\boxd}{\text{dim}_\text{B}}
\newcommand{\aso}{{\text{dim}}_\text{A}}
\newcommand{\qaso}{{\text{dim}}_\text{qA}}
\newcommand{\low}{{\text{dim}}_\text{L}}
\newcommand{\haus}{{\text{dim}}_\text{H}}

\newcommand{\asospec}{{\text{dim}}^{\theta}_\text{A}}

\newcommand{\lset}{L(\Gamma)}
\newcommand{\mups}{\mu_\delta}
\newcommand{\kmin}{\mathbf{k}_{\min}}
\newcommand{\kmax}{\mathbf{k}_{\max}}

\newcommand{\hdist}[2]{d_\mathbb{H}(#1, #2)}

\newcommand{\size}[1]{\vert #1 \vert}

\newcommand{\boldzero}{\mathbf{0}}

\renewcommand{\epsilon}{\varepsilon}

\renewcommand{\geq}{\geqslant}
\renewcommand{\leq}{\leqslant}

\definecolor{lightgray}{rgb}{0.83, 0.83, 0.83}

\title{Refined horoball counting and conformal measure for \\ Kleinian group actions}
\author{Jonathan M. Fraser and Liam Stuart\\ \\
The University of St Andrews, Scotland\\
E-mails: jmf32@st-andrews.ac.uk \& ls220@st-andrews.ac.uk}

\pgfplotsset{compat=1.17}

\begin{document}
\maketitle

\begin{abstract}
Parabolic fixed points form a countable dense subset of the limit set of a non-elementary geometrically finite Kleinian group with at least one parabolic element.  Given such a group, one may associate a standard set of pairwise disjoint horoballs, each tangent to the boundary at a parabolic fixed point.  The diameter  of such a  horoball can be thought of as the `inverse cost' of approximating an arbitrary point in the limit set by the associated   parabolic point. A result of Stratmann and Velani allows one to count horoballs of a given size and, roughly speaking, for small $r>0$ there are $r^{-\delta}$ many horoballs of size approximately $r$, where $\delta$ is the Poincar\'e exponent of the group.  We investigate localisations of this result, where we seek to count horoballs of size approximately $r$ inside a given ball $B(z,R)$.  Roughly speaking, if $r \lesssim R^2$, then we obtain an analogue of the Stratmann-Velani result (normalised by the Patterson-Sullivan measure of $B(z,R)$).  However, for larger values of $r$, the count depends in a subtle way on $z$.

Our counting results have several applications, especially to the geometry of conformal measures supported on the limit set.  For example, we compute or estimate several `fractal dimensions' of certain $s$-conformal measures for $s>\delta$ and use this to examine continuity properties of $s$-conformal measures at $s=\delta$. 
\end{abstract}

\section{Introduction}

\subsection{Kleinian groups, parabolic points and horoballs}

Let $\Gamma$ denote a non-elementary geometrically finite Kleinian group acting on the Poincar\'e ball model of hyperbolic geometry, which models $(d+1)$-dimensional hyperbolic space by $\mathbb{D}^{d+1} = \{z \in \mathbb{R}^{d+1} \mid \size{z} < 1\}$ equipped with the hyperbolic metric $d_{\mathbb{H}}$ defined by
\[ds = \frac{2\size{dz}}{1-\size{z}^2}.\]
We write $\mathbb{S}^d = \{z \in \mathbb{R}^{d+1} \mid \size{z} = 1\}$ to denote the \textit{boundary at infinity} of the space $(\mathbb{D}^{d+1}, d_{\mathbb{H}})$, and letting $\mathbf{0} = (0,\dots, 0) \in \mathbb{D}^{d+1}$, we write $\lset = \overline{\Gamma(\mathbf{0})} \setminus \Gamma(\mathbf{0})$ to denote the limit set of $\Gamma$. We will also make use of the upper half-space model $\mathbb{H}^{d+1} = \mathbb{R}^d \times (0,\infty)$ equipped with the analogous metric, noting that we can move between these models by applying a M\"obius transformation (the Cayley transformation). For more background  on hyperbolic geometry, see \cite{anderson, beardon}.

We will assume throughout that $\Gamma$ contains parabolic elements i.e. maps which fix precisely one point in $\mathbb{S}^d$, and write $P$ for the countable set of parabolic fixed points. It is known (see \cite{stratmannvelani, stratmannurbanski}) that we can associate  a standard set of horoballs (Euclidean balls with interior contained in $\mathbb{D}^{d+1}$ which are tangent to the boundary at a parabolic fixed point) $\{H_p\}_{p \in P}$ such that they are pairwise disjoint, do not contain $\mathbf{0}$, and given any $p \in P$ and $g \in \Gamma$, we have $g(H_p) = H_{g(p)}$. 
Roughly speaking, a Kleinian group $\Gamma$ is \textit{geometrically finite} if it has a fundamental domain with finitely many sides (we refer the reader to \cite{bowditch} for the precise definition).  Write 
\[\delta = \inf \left\{s>0 \mid \sum\limits_{g \in \Gamma} e^{-s\hdist{\boldzero}{g(\boldzero)}} < \infty \right\}\]
to denote the \textit{Poincar\'e exponent} of $\Gamma$. This exponent $\delta$ turns out to be closely related to the dimension theory of $\lset$ and associated measures.  For example, both the Hausdorff and box dimension of $\lset$ are given by $\delta$ in the geometrically finite setting.

\subsection{Notation}

Throughout, we write $A \lesssim B$ if there exists a constant $C>0$ such that $A \leq CB$, and $A \gtrsim B$ if $B \lesssim A$. We write $A \approx B$ if both $A \lesssim B$ and $A \gtrsim B$. The implicit  constants are uniform but may depend on parameters which are fixed throughout the paper, for example the group $\Gamma$.  If we use this notation in a situation where the implicit constants do depend on something more than  $\Gamma$, we will emphasise this explicitly.  For example, if the implicit constant $C$ in $A \lesssim B$ depends on an additional parameter $\alpha$, then we will write $A \lesssim_\alpha B$.

We write $B(z,r)$ to denote the closed (Euclidean) ball centred at $z$ with radius $r>0$.   We write $|X|$ to denote the (Euclidean)  diameter of a non-empty set $X$.  This is not to be confused with $|z|$ which denotes the absolute value of a point $z \in \mathbb{R}^d$.

\subsection{Counting horoballs}

We are interested in counting horoballs of a given size. For example, given a geometrically finite Kleinian group $\Gamma$ and $r>0$, we ask: how many horoballs of diameter approximately $r$ should one expect to see? Stratmann and Velani \cite[Theorem 3]{stratmannvelani} proved the following.
\begin{thm}\label{globhoro}
There exists $\tau \in(0,1)$ such that, for all sufficiently large $k \in \mathbb{N}$, 
\[   \#\left\{p \in P \mid \tau^{k+1} \leq \size{H_p} < \tau^k \right\} \approx \tau^{-k\delta}.\]
\end{thm}
So, from a global point of view, we should expect to see roughly $r^{-\delta}$ horoballs of diameter approximately $r$, provided $r$ is sufficiently small. Our main interest is in refining this result to provide local information.  That is, given $z \in \lset$ and $R>0$, we ask: how many horoballs of diameter approximately $r$ should we expect to find in the ball  $B(z,R)$? As we will see, the answer to this question will depend on the relationship between $r$ and $R$, as well as the proximity of the point $z$ to horoballs of large diameter.

\subsection{The Patterson-Sullivan measure $\mu_\delta$ and the global measure formula}

The limit set  of $\Gamma$ supports a $\Gamma$-ergodic conformal measure   of maximal Hausdorff dimension known as the \emph{Patterson-Sullivan measure}.  This measure was first constructed by Patterson \cite{pattersonlimit} in the Fuchsian case (i.e. the case when $d=1$), and later generalised by Sullivan \cite{sullivandensity} to higher dimensions. Technically speaking, there are a whole family of Patterson-Sullivan measures (see \cite{pattersonclassic}), but as much of the theory for each measure is the same, we simply fix one and discuss \textit{the} Patterson-Sullivan measure.  We denote the Patterson-Sullivan measure by $\mu_\delta$ with   $\delta$ referring to the Poincar\'e exponent of $\Gamma$ (which is also the Hausdorff dimension of $\lset$ and of $\mu_\delta$).

Since we are assuming that $\Gamma$ contains parabolic elements, the Patterson-Sullivan measure does not obey a simple power law (as in, for example, the convex co-compact case), and instead exhibits `parabolic fluctuation'. Stratmann and Velani \cite[Theorem 2]{stratmannvelani} established a \emph{global measure formula} for the Patterson-Sullivan measure $\mu_\delta$, which quantifies this parabolic fluctuation, see  \eqref{GlobThm}.   This is a formula which gives the measure of an arbitrary ball up to uniform constants and has a host of useful applications, for example concerning dimension theory. We recall the statement here, for which we require some notation.  Given $p \in P$, we write $\mathbf{k}(p)$ to denote  the maximal rank of a free abelian subgroup of the stabiliser of $p$, denoted by $\text{Stab}(p) \leq \Gamma$.  In particular, this abelian subgroup    must be generated by $\mathbf{k}(p)$ parabolic elements, noting that  $\text{Stab}(p)$ cannot contain any loxodromic elements, as this would violate discreteness of  $\Gamma$.  It is easy to see that $1 \leq \mathbf{k}(p) \leq d$ for all $p \in P$, and   we define 
\begin{align*}
\kmin &= \min\{\mathbf{k}(p) \mid p \in P\},\\
\kmax &= \max\{\mathbf{k}(p) \mid p \in P\}.
\end{align*}
Let $z \in \lset$ and $T>0$, and define $z_T \in \mathbb{D}^{d+1}$ to be the point on the geodesic ray joining $\mathbf{0}$ and $z$ which is hyperbolic distance $T$ from $\mathbf{0}$. The global measure formula states that 
\begin{equation}\label{GlobThm}
\mups(B(z,e^{-T})) \approx  e^{-T\delta}e^{-\rho(z,T)(\delta-\mathbf{k}(z,T))}
\end{equation}
where $\mathbf{k}(z,T) = \mathbf{k}(p)$ if $z_T \in H_p$ for some $p \in P$ and 0 otherwise, and \[\rho(z,T) = \inf\{\hdist{z_T}{y} \mid y \notin H_p \}\] if $z_T \in H_p$ for some $p \in P$ and 0 otherwise. 
 
\subsection{The $s$-conformal measures $\mu_s$}

An important feature of the Patterson-Sullivan measure is that it is   a \textit{conformal measure} for the action of $\Gamma$. For $s>0$, we say that a Borel probability measure $\mu$ is \textit{s-conformal} for $\Gamma$ if for all  $g \in \Gamma$ and for all Borel measurable $A \subseteq \mathbb{S}^d$,
\[\mu(g(A)) = \int\limits_{A} |g'|^s \rm d \mu. \]
In the case of the Patterson-Sullivan measure, we have $s = \delta$. However, Sullivan \cite{sullivanrelated} was also able to show that, provided $\Gamma$ contains parabolic elements, given any $s>\delta$, there exists an $s$-conformal measure $\mu_s$ supported on $\lset$. These measures exhibit substantially different behaviour than the Patterson-Sullivan measure. For example, \cite[Corollary 20]{sullivandensity} states that these measures must be purely atomic, with atoms at the parabolic fixed points of $\lset$. An important fact which we will use throughout is that given $p \in P$, $\mu_s(\{p\})$ is uniformly comparable to $\size{H_p}^s$ (see \cite[Lemma 3.4]{stratmannvelani}). 

We use our refined horoball counting results to study the geometry of the conformal measures $\mu_s$, for example, proving a `global measure formula', see Theorem \ref{Global2}.  We are especially interested in `continuity properties' of $\mu_s$ at $s=\delta$.   Clearly, $\mu_s$ is \emph{not} continuous at $s=\delta$ in any reasonable topological sense.  For example, $\mu_s$ cannot converge weakly to $\mu_\delta$ as $s \to \delta$ since $\liminf_{s\to \delta} \mu_s(\{p\}) >0 = \mu_\delta(\{p\})$ for all parabolic points $p \in P$.  Therefore, we will probe continuity of $\mu_s$ at $s=\delta$ in a weaker and  more aesthetic sense, see Corollaries \ref{cty1} and \ref{cty2}.

\section{Results and applications}\label{results}
We split this section into several parts, the first part pertaining to our main counting results for horoballs.  Following this we provide three application sections concerning Diophantine approximation, global measure formulae and dimension theory, respectively.

\subsection{Main results: refined horoball  counting}
Our first  result is the following localisation of Theorem \ref{globhoro}.  This result establishes that an appropriately normalised analogue of Stratmann and Velani's global counting result works locally provided one is counting horoballs of size roughly less than the square of the scale of the localisation.
\begin{thm}\label{localvelani}
For all sufficiently small $\tau \in(0,1)$ there exists $C \in (0,1)$ such that for all  $z \in \lset$, all sufficiently small $R>0$ and  all $k \in \mathbb{N}$ such that $\tau^k < CR^{2}$, we have
\[\# \left\{p \in P \cap B(z,R) \mid \tau^{k+1} \leq \size{H_p} < \tau^k \right\} \approx_\tau \tau^{-k\delta} \mups(B(z,R)).\] 
\end{thm}

We defer the proof of Theorem \ref{localvelani} until Section \ref{prooflocalvelani}. This theorem demands further scrutiny of the case when the horoballs are large compared to the scale of the localisation.  First let us observe that we always have a theoretical upper bound  for the local horoball counts, consistent with the precise value obtained in Theorem \ref{localvelani}.

\begin{thm}\label{theoretical}
Let $\tau \in(0,1)$,    $z \in \lset$,  and  $R \in (0,1)$.  If   $k \in \mathbb{N}$ is such that $\tau^{k} \lesssim R$, then
\[\# \left\{p \in P \cap B(z,R) \mid \tau^{k+1} \leq \size{H_p} < \tau^k \right\} \lesssim_\tau  \tau^{-k\delta} \mups(B(z,R)).\] 
Moreover, if  $k \in \mathbb{N}$ is such that $\tau^{k+1} > 2R$, then
\[\# \left\{p \in P \cap B(z,R) \mid \tau^{k+1} \leq \size{H_p} < \tau^k \right\} \leq 1.\] 
\end{thm}

We defer the proof of Theorem \ref{theoretical} until Section \ref{prooftheoretical}.  Clearly the theoretical upper bound given above  is far from being achieved in general once we are beyond the scope of Theorem \ref{localvelani}.  For example, if $z=p$ is a parabolic fixed point and $R$ is sufficiently small compared to $|H_p|$, then there can be no horoballs tangent to points in $B(p,R)$ of diameter much bigger than $R^2$ apart from $H_p$ itself.  This is made precise in Theorem \ref{proximity} below.  The interesting and most subtle case is when $z$ is relatively far from parabolic fixed points associated with large horoballs.  In this case, there may be many `intermediate' horoballs present, but we should not expect as many as the theoretical maximum described in Theorem \ref{localvelani}.

\begin{thm} \label{intermediate}
For all sufficiently small $\tau \in (0,1)$ there exist $c_1>1$ and $c_2 \in (0,1)$ such that for all sufficiently small $R>0$ and all $k \in \mathbb{N}$ such that $R^2< \tau^k < R$ and $z \in L(\Gamma)$ for which there exists $p_0 \in P$ with
\[
c_1 \tau^{k/2} \leq |z-p_0| \leq c_2 \sqrt{R |H_{p_0}|}
\]
we have
\[
\# \left\{ p \in P \cap B(z, R) \mid \tau^{k+1} \leq |H_p| < \tau^k \right\} \gtrsim_\tau \tau^{-k \delta} \mups(B(z,\tau^{k/2})) \left( \frac{R}{\tau^{k/2}}\right)^{\mathbf{k}(p_0)}.
\]
In particular, if $\delta=\kmin=\kmax$, then
\[
\# \left\{ p \in P \cap B(z, R) \mid \tau^{k+1} \leq |H_p| < \tau^k \right\} \approx_\tau  \left(\frac{R}{\tau^{k}}\right)^\delta.
\]
\end{thm}

We defer the proof of Theorem \ref{intermediate} until Section \ref{proofintermediate}. We note that it will always be possible to apply Theorem \ref{intermediate} for some $z \in \lset$ close to a given parabolic fixed point $p_0$.  In particular, for $z_0 \neq p_0$ and $f$ a parabolic transformation fixing $p$, the sequence $f^n(z_0)$ approaches $p_0$ at a polynomial rate ($\approx 1/n$ with implicit constants depending on $z_0$ and $f$).  Therefore if $R$ is sufficiently small compared with $|H_{p_0}|$, and we choose $k$ to ensure that the interval we need to bound $|z-p_0|$ within is sufficiently large compared to the square of the right end of that interval, that is,
\[
\frac{ c_2 \sqrt{R |H_{p_0}|} - c_1 \tau^{k/2}}{ c_2^2R |H_{p_0}|}
\]
is sufficiently large, then we can choose $n$ such that 
\[
c_1 \tau^{k/2} \leq |f^n(z_0)-p_0| \leq c_2 \sqrt{R |H_{p_0}|}.
\]
Finally, we provide a result in the other direction.  If $z$ is relatively too close to a parabolic point associated with a large horoball $|H_p|$, then there cannot be other large horoballs contributing to the local count apart from possibly $H_p$ itself.

\begin{thm} \label{proximity}
Let $\lambda \in (1,2]$, $c \geq 1$, $z \in L(\Gamma)$, and $R>0$.  If there exists $p_0 \in P$ with
\[
c|H_{p_0}| R^\lambda \geq c^2R^{2\lambda}+(|z-p_0|+R)^2,
\]
then
\[
\# \left\{ p \in P \cap B(z, R) \mid  cR^\lambda \leq |H_p| \right\}  \leq 1.
\]
\end{thm}

We defer the proof of Theorem \ref{proximity} until Section \ref{proofproximity}.  It is useful to note that the assumption in Theorem \ref{proximity} guarantees 
\[
\sqrt{c} R^{\lambda/2} >|z-p_0|
\]
which forbids the assumption of Theorem \ref{intermediate} with $\tau^k = R^\lambda$ and $c_1=\sqrt{c}$.  Moreover, the assumption in Theorem \ref{proximity} always holds when $z=p \in P$ and $R$ is small enough in terms of $|H_p|$,  $\lambda$ and $c$.  When $\lambda <2$, we can take $c=1$ but for $\lambda = 2$, we need $c > |H_{p_0}| ^{-1}$.

\subsection{Application: Diophantine approximation and counting rationals}

Diophantine approximation is traditionally the study of how well real numbers are approximated by rationals.  There are well-established links between this and hyperbolic geometry, especially via the modular group $\textup{PSL}(2,\mathbb{Z})$.  For additional information on (the more general theory of) Diophantine approximation on Kleinian groups, we refer the reader to \cite{diophantine, stratmanndio}.  For example,  an application of \cite[Theorem DT']{diophantine} recovers the one-dimensional case of Dirichlet's Theorem.

Consider the upper half-plane model $\mathbb{H}^2$, and the action of the Kleinian group $\Gamma = \textup{PSL}(2,\mathbb{Z})$ on $\mathbb{H}^2$. It is an easy exercise to show that $\lset = \mathbb{R} \cup \{\infty\}$, $\delta=\kmax=\kmin =1$,  $P = \mathbb{Q} \cup \{\infty\}$, and $\size{H_{p/q}} \approx 1/q^2$  for  $p/q \in P$ with $p$, $q$ coprime integers.    Thus, through applications of Theorems \ref{localvelani}, \ref{intermediate}, and \ref{proximity}, we get the following results which may be expressed purely in terms of real numbers (that is, having nothing \emph{a priori} to do with hyperbolic geometry).  Corollary \ref{diocorollary} and \ref{diocorollary3} are straightforward to derive   directly using basic Diophantine properties of rationals, but we include their statements as simple examples of how our work can be applied.  Corollary \ref{diocorollary2} is perhaps more interesting in its own right.

\begin{cor}\label{diocorollary}
For all sufficiently small $\tau \in(0,1)$, there exists $C \in (0,1)$ such that for all sufficiently small $R>0$, all $z \in \mathbb{R}$, and for all $k \in \mathbb{N}$ such that $\tau^k < CR^2$, we have
\[\sum_{\substack{q \in \mathbb{N} :\\ \tau^{-k} < q^2  \leq \tau^{-k-1}}} \#\left\{p \in \mathbb{Z} \mid \gcd(p,q) = 1, \ |p/q-z| \leq R\right\} \approx_\tau \tau^{-k}R.  \]
\end{cor}

\begin{cor}\label{diocorollary2}
For all sufficiently small $\tau \in (0,1)$ there exist $c_1>1$ and $c_2 \in (0,1)$ such that for all sufficiently small $R>0$ and all $k \in \mathbb{N}$ such that $R^2< \tau^k < R$ and $z \in \mathbb{R}$ for which there exist coprime $p_0 \in \mathbb{Z}$ and $q_0 \in \mathbb{N}$  with
\[
c_1 \tau^{k/2} \leq |z-p_0/q_0| \leq c_2 q_0^{-1} \sqrt{R} 
\]
we have
\[
\sum_{\substack{q \in \mathbb{N} :\\ \tau^{-k} < q^2  \leq \tau^{-k-1}}}\# \left\{ p \in \mathbb{Z}  \mid \gcd(p,q)=	1, \ |p/q-z| \leq R \right\} \approx_\tau \tau^{-k}  R.
\]
\end{cor}

\begin{cor} \label{diocorollary3}
Let $\lambda \in (1,2]$,  $c \geq 1$, $z \in \mathbb{R}$, and $R>0$.  If there exist coprime $p_0 \in \mathbb{Z}$ and $q_0 \in \mathbb{N}$ such that
\[
c\frac{R^\lambda}{q_0^2} \geq c^2 R^{2\lambda}+(|z-p_0/q_0|+R)^2,
\]
then
\[
\# \left\{ p \in \mathbb{Z},\  q \in \mathbb{N} \mid \gcd(p,q)=1, \ |p/q-z|\leq R, \ cR^\lambda \leq 1/q^2 \right\}  \leq 1.
\]
\end{cor}

\subsection{Application: global measure formulae for conformal measures}

We are interested in obtaining global measure formulae for the $s$-conformal measures $\mu_s$.  These formulae will of course be rather different from \eqref{GlobThm} for $s>\delta$ since these measures are purely atomic.  Using Theorem \ref{localvelani}, we   obtain the following global measure formula for the measures $\mu_s$.

\begin{thm}\label{Global2}
Fix $\tau$ and $C$ permitted by Theorem \ref{localvelani}.  Let $z \in \lset$ and $R>0$ be sufficiently small. Then
\begin{equation*}
\mu_s(B(z,R)) \approx_\tau  R^{2(s-\delta)} \mups(B(z,R)) \ + \  \sum\limits_{\substack{k \in \mathbb{N}: \\ {CR^2 \leq\tau^{k} < R}}} \#\{ p \in P \cap B(z,R) \mid  {\tau^{k+1} \leq \size{H_p} < \tau^{k}}\} \, \tau^{ks} \ \  + \  \size{H_{p'}}^s
\end{equation*}
where $p' \in P \cap B(z,R)$ is chosen to maximise $\size{H_{p'}}$ given that $\size{H_{p'}} \geq \tau^k $ for all integers  $k$ with $\tau^{k} < R$. If no such $p'$ exists then we take $\size{H_{p'}}^s =0$.
\end{thm}

\begin{proof}
Since $\mu_s$ is purely atomic with atoms at points in $P$, we get
\begin{align}\label{purely}
\mu_s(B(z,R)) &= \sum_{p \in P \cap B(z,R)} \mu_s(\{p\}) \nonumber \\
              &= \sum\limits_{\substack{k \in \mathbb{N}: \\ {\tau^{k} < CR^2}}} \sum\limits_{\substack{p \in P \cap B(z,R): \\ {\tau^{k+1} \leq \size{H_p} < \tau^{k}}}} \mu_s(\{p\}) +
              \sum\limits_{\substack{k \in \mathbb{N}: \\ {CR^2 \leq\tau^{k} < R}}} \sum\limits_{\substack{p \in P \cap B(z,R): \\ {\tau^{k+1} \leq \size{H_p} < \tau^{k}}}} \mu_s(\{p\}) + \mu_s(\{p'\}).
\end{align}
Applying Theorem \ref{localvelani} to the first term gives us
\begin{equation*}
  \sum\limits_{\substack{k \in \mathbb{N}: \\ {\tau^{k} < CR^2}}} \sum\limits_{\substack{p \in P \cap B(z,R) \\ {\tau^{k+1} \leq \size{H_p} < \tau^{k}}}} \mu_s(\{p\}) 
  \approx_\tau  \sum\limits_{\substack{k \in \mathbb{N}: \\ {\tau^{k} < CR^2}}} \tau^{k(s-\delta)} \mups(B(z,R))
  \approx_\tau R^{2(s-\delta)} \mups(B(z,R))
\end{equation*}
and so (\ref{purely}) becomes
\begin{align*}
&R^{2(s-\delta)} \mups(B(z,R)) +  \sum\limits_{\substack{k \in \mathbb{N}: \\ {CR^2 \leq\tau^{k} < R}}} \sum\limits_{\substack{p \in P \cap B(z,R) : \\ {\tau^{k+1} \leq \size{H_p} < \tau^{k}}}} \mu_s(\{p\}) + \mu_s(\{p'\}) \\
\approx_\tau \ &R^{2(s-\delta)} \mups(B(z,R)) \ + \  \sum\limits_{\substack{k \in \mathbb{N}: \\ {CR^2 \leq\tau^{k} < R}}} \#\{ p \in P \cap B(z,R) \mid  {\tau^{k+1} \leq \size{H_p} < \tau^{k}}\} \, \tau^{ks} \ \  + \  \size{H_{p'}}^s,
\end{align*}
as required.
\end{proof}

A few remarks are in order.  Firstly, this global measure formula expresses the measure of an arbitrary ball as the sum of three terms.  The first term
\[
R^{2(s-\delta)} \mups(B(z,R))
\]
is the only term which is always non-zero and corresponds to parabolic points inside $B(z,R)$ with associated horoballs having diameter at most a constant multiple of $R^2$.  This may, as we shall see, take care of all of the measure of $B(z,R)$.  The first term also involves  $\mups(B(z,R))$, and so we may apply (\ref{GlobThm}) to re-express this in terms of $R$ and the familiar parabolic fluctuation function $\rho(z,T) = \rho(z,-\log R)$. 

The second term
\[
 \sum\limits_{\substack{k \in \mathbb{N}: \\ {CR^2 \leq\tau^{k} < R}}} \#\{ p \in P \cap B(z,R) \mid  {\tau^{k+1} \leq \size{H_p} < \tau^{k}}\} \, \tau^{ks} 
\]
is the most subtle and corresponds to parabolic points inside $B(z,R)$ with associated horoballs having diameter at least a constant multiple of $R^2$ and at most $R$.  There may be none of these and so this term can drop out.  Moreover, even if there are some but the associated horoballs  have diameter $\lesssim R^2$, then the second term can be subsumed into the first. If the second term  is non-zero, it is useful to observe that it can be bounded from above by
\[
\lesssim_\tau R^{s-\delta} \mups(B(z,R)).
\]
This uses Theorem \ref{theoretical} and then sums the truncated geometric series. In particular, this bound would dominate the first term if it were obtained.  Moreover, we may use our results in Theorems \ref{intermediate} and \ref{proximity} to estimate the second term if the assumptions are satisfied. 

 The third term 
\[
 \size{H_{p'}}^s
\]
may again drop out.  However, if it is present, then we can bound it by
\[
R^s \lesssim_\tau  \size{H_{p'}}^s \leq  1
\]
and these bounds clearly cannot be improved.  If the third term is present, then it will dominate the first term provided $s> \kmax$ and the theoretical upper bound for the second term provided $\delta > \kmax$.

 As a consequence of our global measure formula we obtain a `continuity type result' for $\mu_s$ at $s=\delta$.  Roughly speaking, it says that the $\mu_s$ measure of a small enough ball is uniformly comparable to the $\mu_\delta$ measure of the same ball up to at most an atom as $s \to \delta$.

\begin{cor} \label{cty1}
There exists a  uniform constant $A  \geq 1$ such that for all $z \in \lset$ and $R>0$   sufficiently small either
\[
A^{-1} \leq \liminf_{s \to \delta} \frac{\mu_s(B(z,R))}{ \mu_\delta(B(z,R))} \leq \limsup_{s \to \delta} \frac{\mu_s(B(z,R))}{ \mu_\delta(B(z,R))} \leq A
\]
or there exists a point $p \in P \cap B(z,R)$ and a weight $\alpha=\alpha(z,R,s)>0$ such that
\[
A^{-1} \leq \liminf_{s \to \delta} \frac{( \mu_s-\alpha \Delta_p)(B(z,R))}{ \mu_\delta(B(z,R))} \leq \limsup_{s \to \delta} \frac{( \mu_s-\alpha \Delta_p)(B(z,R))}{ \mu_\delta(B(z,R))} \leq A,
\]
where $\Delta_p$ is a unit point mass at $p$.
\end{cor}

\begin{proof}
By Theorem \ref{Global2}
\[
R^{2(s-\delta)} \mups(B(z,R)) \lesssim \mu_s(B(z,R)) \lesssim R^{ s-\delta} \mups(B(z,R))
\]
provided there does not exist $p \in B(z,R) \cap P$ with $|H_p| >2 R$.  In this case the first  result holds.  However, if there   does   exist $p \in B(z,R) \cap P$ with $|H_p| >2 R$, then there is only one such $p$ and
\[
R^{2(s-\delta)} \mups(B(z,R)) \lesssim (\mu_s-\mu_s(\{p\}) \Delta_p) (B(z,R)) \lesssim R^{ s-\delta} \mups(B(z,R))
\]
and the second  result holds with $\alpha = \mu_s(\{p\})$.
\end{proof}

\subsection{Application: dimension theory of conformal measures}

For a more detailed analysis on the dimensions of measures, we refer the reader to \cite{techniques, fraserbook}. We write $\haus$ for the Hausdorff dimension of a set or measure, but omit the definition.  Let $\mu$ denote a finite compactly supported Borel measure on $\mathbb{R}^d$, and let $F = \text{supp}(\mu) = \{x \in \mathbb{R}^d \mid \mu(B(x,r)) > 0 \ \text{for all} \ r>0 \}$ denote the \textit{support} of $\mu$. The \textit{upper box dimension of $\mu$} is given by
\begin{align*}
\ubox \mu = \inf \Big\{ s \mid  \exists C>0 \ : \ \forall \ 0<r< \vert F \vert \  :  \ \forall x \in F \ : \  \mu(B(x,r)) \geq Cr^{s} \Big\} 
\end{align*}
and the \textit{lower box dimension of $\mu$} is given by
\begin{align*}
\lbox \mu = \inf \Big\{ s \mid \exists  C>0 \  : \ \forall  r_0>0 \ : \  \exists \ 0<r<r_0 \ :\ \forall \ x\in F \ :  \  \mu(B(x,r)) \geq Cr^{s} \Big\}.  
\end{align*}
If $\ubox \mu = \lbox \mu$, then we write $\boxd \mu$ to denote the common value and call it the \textit{box dimension of $\mu$}. The \textit{Assouad dimension of $\mu$} is defined by
\begin{align*}
\aso \mu = \inf \Bigg\{ s \geq 0 \mid  \exists C>0 \ : \ \forall \ 0<r<R< \vert F \vert \  :  \ \forall x \in F \ : \ \frac{\mu(B(x,R))}{\mu(B(x,r))} \leq C \left(\frac{R}{r} \right)^{s} \Bigg\}.
\end{align*}
The Assouad spectrum of $\mu$ attempts to understand the gap between the box and Assouad dimensions of a measure.  For $\theta \in (0,1)$,  the \textit{Assouad spectrum of $\mu$} is defined by
\begin{align*}
\asospec \mu = \inf \Bigg\{ s \geq 0 \mid  \exists C>0 \ : \ \forall \ 0<r< \vert F \vert \  :  \ \forall x \in F  \ : \ \frac{\mu(B(x,r^\theta))}{\mu(B(x,r))} \leq C \left(\frac{r^\theta}{r} \right)^{s} \Bigg\}. 
\end{align*}
The Assouad spectrum is continuous in $\theta$.  Moreover, the limit as $\theta \to 1$ is known to exist, and is referred to as the \textit{quasi-Assouad dimension}, denoted by $\qaso \mu$, and the limit as  $\theta \to 0$ is the upper box dimension of $\mu$.   We can relate the above notions of dimension as follows:  
\[\haus \mu \leq \lbox \mu \leq \ubox \mu \leq \asospec \mu \leq \qaso \mu \leq \aso \mu\]
 for all $\theta \in (0,1)$.

We are interested in the dimension theory of the $s$-conformal measures $\mu_s$.  Certain dimensions are not well-suited to this problem.  For example, it is immediate using the fact that $\mu_s$ is purely atomic for $s>\delta$ that $\mu_s$ is exact dimensional with dimension 0 and so the Hausdorff, packing and entropy dimensions of $\mu_s$ are 0 for all $s>\delta$.  Moreover, the Assouad dimension of $\mu_s$ is $\infty$ for all $s>\delta$. This is less trivial, but it follows from   a result of Kaufman and Wu \cite[Lemma 2]{kaufmanjang}  that $\mu_s$ is not doubling, which gives $\aso \mu_s = \infty$ (see \cite[Lemma 3.2]{jarvenpaapacking}).  The lower dimension and lower spectrum, which are natural duals to the Assouad dimension and Assouad spectrum, are also not well-suited.  In fact, any measure with an atom immediately has lower spectrum and lower dimension identically 0. In contrast, the box dimension and Assouad spectrum are perfectly suited to studying $\mu_s$.

\begin{thm}\label{dimensions}
For all $s> \delta$,
\begin{align*}
\max\{s, 2s-\kmin\} \leq  \normalfont{  \lbox} \mu_s  \leq  \normalfont{  \ubox} \mu_s \leq  \max\{2s-\delta, 2s-\kmin\} 
\end{align*}
Moreover, for $\theta \in (0,1/2]$, 
\begin{align*}
\frac{\max\{s, 2s-\kmin\}}{1-\theta} \leq  \normalfont{ \asospec} \mu_s \leq  \frac{\max\{2s-\delta, 2s-\kmin\}}{1-\theta}
\end{align*}
and for $\theta \in ( 1/2, 1)$,
\begin{align*}
\frac{ \max\{2\theta (s-\kmin)+\kmin, 2s-\kmin\} }{1-\theta} \leq  \normalfont{ \asospec} \mu_s &\leq  \frac{\max\{2s-\delta, 2s-\kmin\}}{1-\theta}.
\end{align*}
In particular,   $\normalfont{ \qaso} \mu_s    = \infty$. 
\end{thm}
We defer the proof of Theorem \ref{dimensions} until Section \ref{proofdimensions}. Note that the above estimates become precise formulae provided $\delta \geq \kmin$.
\begin{cor}
If $s > \delta \geq \kmin$, then
\begin{align*}
\normalfont{\lbox} \mu_s = \normalfont{\ubox} \mu_s &= 2s-\kmin \\
\text{and } \hspace{1.2cm} \normalfont{\asospec} \mu_s &= \frac{2s-\kmin}{1-\theta}
\end{align*}
for all $\theta \in (0,1)$.
\end{cor}

We do not know if the bounds from Theorem \ref{dimensions} are sharp in the case $\delta<\kmin$.  It would be interesting to consider this further. We established that $\qaso \mu_s = \infty$, but it would be interesting to know whether or not this holds for arbitrary  purely atomic measures with perfect support.  Indeed, it follows from the result of  Kaufman and Wu that $\aso \nu = \infty$ for such measures $\nu$.

\begin{ques}
Let $(X,\sigma)$ be a compact metric space and $\nu$ be a Borel measure on $X$ with $\normalfont{\qaso} \nu < \infty$.  Is it true that $\nu(\{x\}) = 0$ for every accumulation point $x$ in the support of $\nu$?
\end{ques}

As discussed above, we are especially interested in   `continuity properties' of $\mu_s$ at $s=\delta$.    One approach to this problem is to check if various fractal dimensions vary continuously at $s=\delta$. This fails in a very dramatic way for the Hausdorff,  Assouad and quasi-Assouad dimensions for example since $\haus \mu_\delta= \delta>0$ and $\aso \mu_\delta = \qaso \mu_\delta = \max\{\kmax, 2\delta-\kmin\}<\infty$. For discussion regarding the various dimensions of the Patterson-Sullivan measure, we refer the reader to \cite{fraserregularity,fraserstuart}. More interestingly,  for $\theta \in (0,1/2]$, 
 \[
\liminf_{s \to \delta} \normalfont{ \asospec} \mu_s  \geq   \frac{\max\{ \delta, 2\delta-\kmin\}}{1-\theta}  > \normalfont{ \asospec } \mu_\delta
 \]
 and so $ \normalfont{ \asospec} \mu_s$ is \emph{not} continuous at $s=\delta$, despite being positive and finite for all $s \geq \delta$.  However,
\[
\normalfont{  \ubox} \mu_s    \to \max\{  \delta, 2\delta-\kmin\}  = \normalfont{  \boxd} \mu_\delta
\]
and
\[
\normalfont{  \lbox} \mu_s    \to \max\{  \delta, 2\delta-\kmin\}  = \normalfont{  \boxd} \mu_\delta
\]
as $s \to \delta$ and so we get another continuity type result, this time cast in the language of dimension.

\begin{cor} \label{cty2}
The functions $s \mapsto \normalfont{  \ubox} \mu_s $ and  $s \mapsto \normalfont{  \lbox} \mu_s $ defined for $s \in [\delta, \infty)$   are continuous at $s=\delta$. 
If $\delta \geq \kmin$, then  $s \mapsto \normalfont{  \boxd} \mu_s $ is well-defined and affine for all  $s \in [\delta, \infty)$.
\end{cor}

\section{Remaining proofs}\label{proofs}

\subsection{Proof of Theorem \ref{localvelani}} \label{prooflocalvelani}
Fix $\tau \in (0,1)$ which later we will see must be chosen sufficiently small depending only on the group $\Gamma$. Let $C_1R^{2}>r>0$, where the constant $C_1 \in (0,1)$ is to be determined, and let $z \in \lset$. We first recall that due to \cite[Theorem 1]{stratmannvelani}, there exists a constant $\kappa>0$ such that for sufficiently small $r>0$, 
\begin{equation} \label{squeezecover} \lset \subseteq \bigcup\limits_{\substack{p \in P: \\ {\size{H_p}\geq r}}} \Pi\left(\kappa \sqrt{\frac{r}{\size{H_p}}}H_p\right)
\end{equation}
with multiplicity $\lesssim 1$.   Here $\Pi$ denotes the radial projection from $\mathbf{0}$ onto the boundary $\mathbb{S}^d$ and $\lambda H_p$ ($\lambda>0$) denotes the `squeezed horoball' tangent to the boundary at $p$ but with diameter    $\lambda |H_p|$.  For notational convenience, we write $\lambda_p = \kappa\sqrt{r/{\size{H_p}}}$.  Therefore, for sufficiently small $R$, we have
\[\mups(B(z,R)) \lesssim \sum\limits_{\substack{p \in P\cap B(z,R): \\ {r \leq \size{H_p}<R}}} \mups\left(\Pi\left(\lambda_p H_p\right)\right) \ + 
\sum\limits_{\substack{p \in P: \\ \Pi(\lambda_p H_p)\cap B(z,R) \neq \emptyset \\ {R\leq\size{H_p}}}} \mups\left(\Pi\left(\lambda_p H_p \right) \cap B(z,R)\right).
\]
A simple disjointness argument shows that the number of $p \in P$ satisfying the conditions of the second sum must be $\leq K_d$, where $K_d$ is a constant dependant only on $d$. Furthermore, note that for all $p \in P$, we have $\size{\Pi\left(\lambda_p H_p \right)} \lesssim \sqrt{r} < \sqrt{C_1}R$, and so we can choose $C_1$ dependant on the implied constants such that $\size{\Pi\left(\lambda_p H_p \right)}<R$. 

The lower dimension, denoted by $\low$, is a notion dual to the Assouad dimension.  We refer to \cite{fraserbook} for the definition of lower dimension, but for our purposes we use that when the lower dimension of a measure is positive, one can effectively bound from below the ratio of the measure of concentric balls.  In particular, it was shown in \cite{fraserregularity} that $\low \mups>0$. Therefore, by choosing $p \in P$ satisfying the conditions of the second sum and an appropriate $z' \in \Pi\left(\lambda_p H_p \right) \cap B(z,R)$, for fixed $0 < \epsilon < \low\mups$, we have
\begin{align*}
\frac{\mups(B(z,R))}{\mups\left(\Pi\left(\lambda_p H_p \right) \cap B(z,R)\right)} \gtrsim \frac{\mups(B(z',R))}{\mups\left(\Pi\left(\lambda_p H_p \right) \cap B(z,R)\right)}
&\gtrsim \left(\frac{R}{\size{\Pi\left(\lambda_p H_p \right)}}\right)^{\low\mups-\epsilon} \\
& \gtrsim \left(\frac{1}{\sqrt{C_1}}\right)^{\low\mups-\epsilon} 
\end{align*}
where we have used the fact that $\mups$ is doubling. In particular, we can also ensure $C_1$ is chosen (dependant on the implied constants) sufficiently small to ensure
\begin{equation}\label{lowerestimate}
\frac{\mups(B(z,R))}{\mups\left(\Pi\left(\lambda_p H_p \right) \cap B(z,R)\right)} \geq 100K_d.
\end{equation}
Also, using the multiplicity in the cover \eqref{squeezecover} and the fact that $\mups$ is doubling, we get
\[\mups(B(z,R)) \gtrsim \mups(B(z,2R)) \gtrsim \sum\limits_{\substack{p \in P\cap B(z,R): \\ {r \leq \size{H_p}<R}}} \mups\left(\Pi\left(\lambda_p H_p\right)\right). \]
Combining this with (\ref{lowerestimate}) gives
\begin{align*}
\sum\limits_{\substack{p \in P\cap B(z,R) : \\ {r \leq \size{H_p}<R}}} \mups\left(\Pi\left(\lambda_p H_p\right)\right) \lesssim \mups(B(z,R)) 
&\lesssim \sum\limits_{\substack{p \in P\cap B(z,R): \\ {r \leq \size{H_p}<R}}} \mups\left(\Pi\left(\lambda_p H_p\right)\right) \ + \  \mups(B(z,R))/100
\end{align*}
and so, using \cite[Corollary 3.5]{stratmannvelani} which gives a fomula for the $\mu_\delta$ measure of the radial projection of a squeezed horoball,  we have
\begin{equation}\label{counthoro}
 \mups(B(z,R))  \approx \sum\limits_{\substack{p \in P\cap B(z,R): \\ {r \leq \size{H_p}<R}}} \mups\left(\Pi\left(\lambda_p H_p\right)\right) \approx r^{\delta} \sum\limits_{\substack{p \in P\cap B(z,R): \\ {r \leq \size{H_p}<R}}} \left(\frac{\size{H_p}}{r}\right)^{\mathbf{k}(p)/2}.
\end{equation}
With this estimate in place, we now follow the argument from \cite[Proof of Theorem 3]{stratmannvelani}. We   wish to consider $0<r<C_1 C_2 R^2$, where $C_2$ is some constant to be determined. Explicitly, (\ref{counthoro}) implies the existence of a constant $C_3\geq1$ such that
\begin{equation}\label{explicit}
 \mups(B(z,R))/C_3 \leq r^{\delta} \sum\limits_{\substack{p \in P\cap B(z,R): \\ {r \leq \size{H_p}<R}}} \left(\frac{\size{H_p}}{r}\right)^{\mathbf{k}(p)/2} \leq C_3\mups(B(z,R)).   
\end{equation}
Note that from our arguments above, $C_3$ can be chosen independently of $\tau,z,R$ and $r$, as none of the implied constants used in the derivation of (\ref{counthoro}) depend on these quantities. Furthermore, for $1 \leq \alpha < C_1R^{2}/r$,
\[\mups(B(z,R))/C_3 \leq (\alpha r)^{\delta} \sum\limits_{\substack{p \in P\cap B(z,R): \\ {\alpha r \leq \size{H_p}<R}}} \left(\frac{\size{H_p}}{\alpha r}\right)^{\mathbf{k}(p)/2} \leq C_3\mups(B(z,R))\]
which implies
\begin{align*}
\sum\limits_{\substack{p \in P\cap B(z,R) : \\ {r \leq \size{H_p}<\alpha r}}} \left(\frac{\size{H_p}}{r}\right)^{\mathbf{k}(p)/2} &= \sum\limits_{\substack{p \in P\cap B(z,R) : \\ {r \leq \size{H_p}<R}}} \left(\frac{\size{H_p}}{r}\right)^{\mathbf{k}(p)/2} - \sum\limits_{\substack{p \in P\cap B(z,R): \\ {\alpha r \leq \size{H_p}<R}}} \left(\frac{\size{H_p}}{r}\right)^{\mathbf{k}(p)/2} \\
&\geq r^{-\delta}(1/C_3-C_3\alpha^{\kmax/2-\delta})\mups(B(z,R)).
\end{align*}
We now choose $C_2$ to be small enough to ensure that we can make $\alpha$ sufficiently large, in particular, such that
\[\alpha^{-\kmax/2}(1/C_3-C_3\alpha^{\kmax/2-\delta})\geq C_4\]
is satisfied for some constant $C_4=C_4(\alpha)>0$. Recall here the well-known estimate $\delta>\kmax/2$.  Combined with (\ref{explicit}), we get
\[\sum\limits_{\substack{p \in P\cap B(z,R) : \\ {r \leq \size{H_p}<\alpha r}}} 1 \approx_\alpha r^{-\delta}\mups(B(z,R))\]
and the result follows provided we can choose $\alpha = 1/\tau$.  However, this will be possible for $\tau$ sufficiently small depending only on $C_3$ and other fixed constants.  Finally, $C$ is chosen (depending on $\tau$) to be $C=C_1C_2$.

\subsection{Proof of Theorem \ref{theoretical}} \label{prooftheoretical}
Let $\tau \in(0,1)$,    $z \in \lset$,  and  $R \in (0,1)$.  If   $k \in \mathbb{N}$ is such that $\tau^{k} \lesssim R$, then
\begin{align*}
\mups(B(z,R)) &\geq \sum_{  \substack{p \in P \cap B(z,R) :\\
 \tau^{k+1} \leq \size{H_p} < \tau^k} } \mups (B(p,\tau^{k+1}/10) \cap B(z,R) ) \\
 &\gtrsim_\tau
 \sum_{\substack{p \in P \cap B(z,R) :\\
 \tau^{k+1} \leq \size{H_p} < \tau^k} } \mups (B(p,\size{H_p}) \cap B(z,R) ) 
 \gtrsim_\tau \tau^{k\delta} \sum_{  \substack{p \in P \cap B(z,R) :\\
 \tau^{k+1} \leq \size{H_p} < \tau^k} } 1
 \end{align*}
and the result follows.  In the above we used \eqref{GlobThm}, that $\mups$ is doubling, and that, since  $\tau^{k} \lesssim R$,  we can ensure  that $B(p,\size{H_p}) \cap B(z,R)$ contains  a ball of radius $\gtrsim \size{H_p}$ for the $p$ we sum over. 

 If  $k \in \mathbb{N}$ is such that $\tau^{k+1} > 2R$, then since horoballs $H_p$ ($p \in P$) are pairwise disjoint, it is immediate that
\[\# \left\{p \in P \cap B(z,R) \mid \tau^{k+1} \leq \size{H_p} < \tau^k \right\} \leq 1.\]

\subsection{Proof of Theorem \ref{intermediate}} \label{proofintermediate}
We switch to the upper-half space model $\mathbb{H}^{d+1} = \mathbb{R}^d \times (0,\infty)$ and may assume that $p_0=\mathbf{0} \in P$. Then $\mathbf{0}$ has a set of $\mathbf{k}(\mathbf{0})$ parabolic elements  which all fix $\mathbf{0}$ and generate a free abelian group of rank $\mathbf{k}(\mathbf{0})$.  We conjugate $\mathbf{0}$ to $\infty$ by applying the isometry   $\iota$ which inverts in  the sphere centred at $\mathbf{0}$ with radius 1. Note that $\iota$ does not preserve orientation, but this is not an issue. Conjugating the parabolic elements fixing $\mathbf{0}$ by $\iota$ results in a collection $\{f_1 , \dots ,f_{\mathbf{k}(\mathbf{0})} \}$ of parabolic elements of the form
\[f_i(z) = A_iz+t_i\]
for some $t_i \in \mathbb{R}^d$ and some finite order orthogonal  matrix $A_i$  acting on the boundary  $\mathbb{R}^d$. (Formally this only defines $f$ on the boundary, but this is extended to $\mathbb{H}^{d+1} $ by $f(z,\omega)=(f(z),\omega)$.)  Further, by taking sufficiently large powers of the $f_i$, we may assume without loss of generality that they all take the form $f_i(z) = z+t_i$.   Moreover, we may assume that
\[
|t_i| \approx |H_{p_0}|^{-1}
\]
for all $i$, where the implicit constants may depend on the matrices $A_i$.  This can be seen by restricting to a 2-dimensional slice through $\mathbb{H}^{d+1}$ which is stabilised by $f_i$ (namely, the span of $t_i$ and a vector normal to the boundary) and then applying the following 2-dimensional argument.

\begin{lem}
Suppose $d=1$ and $0 \in P \subseteq \mathbb{R} \cup \{\infty\}$.  Then there exists $f \in \Gamma \leq \textup{PSL}(2, \mathbb{R})$ and  $\alpha > 0$ such that for all $z \in \mathbb{H}^2$, 
\[f(z) = \frac{z}{\alpha z+1} \]
where $\alpha \approx \size{H_{0}}^{-1}$. In particular, conjugating by the circle inversion $\iota$ given by $\iota(z) = 1/z$ gives $\iota f \iota^{-1} = z+\alpha$. 
\end{lem}
\begin{proof}
We have $g(p') =0$ and  $H_{0} = H_{g(p')} = g(H_{p'})$  for one of the finitely many inequivalent parabolic points $p'$ and some $g \in \Gamma$. Since there are only finitely many  inequivalent parabolic points, we may assume that $p'=\infty$ and $H_{p'} = \{ x+ai : x \in \mathbb{R}\}$ for some uniform constant $a>0$. Since $g(\infty) = 0$,
\[
g(z) = \frac{u}{ z+v}
\]
for some $u, v \in \mathbb{R} \setminus\{0\}$ with $u<0$.  Moreover, $\infty$ is fixed by a parabolic element which we may assume is of the form $h(z) = z+b$ for some uniform constant $b>0$.   Then, by direct calculation, we have
\[
g h^{-1} g^{-1}(z) =    \frac{z}{ (-b/u)z+1}
\]
and so we can take $\alpha = -b/u$ and $f = gh^{-1}g^{-1}$.  Then
\begin{align*}
 \size{H_{0}} = \size{g(H_{p'})} = \sup_{z \in H_{p'}} \textup{Im} \left( \frac{u}{ z+v}\right) = \sup_{x \in \mathbb{R}}  \textup{Im}\left( \frac{u}{ x+ai+v}\right) = \sup_{x \in \mathbb{R}} \frac{-au}{(x+v)^2+a^2} = \frac{-u}{a} \approx 1/\alpha
\end{align*}
as required.
\end{proof}

Before proceeding, we also need a small modification of Theorem \ref{localvelani}.    
\begin{lem}\label{tauball}
Fix $C$   as in Theorem \ref{localvelani}.  For sufficiently small  $\tau \in (0,1)$,  all $z \in \lset$ and all sufficiently large $k \in \mathbb{N}$, we have
\[\# \left\{p \in P \cap B(z,(2/\sqrt{C})\tau^{k/2}) \mid 2\tau^{k+1} \leq \size{H_p} < \tau^k/2 \right\} \approx_\tau \tau^{-k\delta} \mups(B(z,\tau^{k/2})).\]
\end{lem}
\begin{proof}
The $\lesssim_\tau$ direction of the lemma is immediate from  Theorem \ref{localvelani}. For the $\gtrsim_\tau$ direction, let $\tau_0$ be a suitable value of $\tau$ from the statement of  Theorem \ref{localvelani} such that $\tau_0 \leq 1/2$, and choose $\tau = \tau_0^3$. Fix $z \in \lset$. Note that the ball $B(z,(2/\sqrt{C})\tau_0^{(k-1)/2})$ is large enough to directly apply Theorem \ref{localvelani} to obtain 
\[\# \left\{p \in P \cap B(z,(2/\sqrt{C})\tau_0^{(k-1)/2}) \mid \tau_0^{k+1} \leq \size{H_p} < \tau_0^k \right\} \approx_\tau \tau_0^{-k\delta} \mups(B(z,\tau_0^{k/2}))\]
for  sufficiently large $k \in \mathbb{N}$. Replacing $k$ by $3k+1$ and using $\tau = \tau_0^3$, this further gives
\[\# \left\{p \in P \cap B(z,(2/\sqrt{C})\tau^{k/2}) \mid \tau^{k+2/3} \leq \size{H_p} < \tau^{k+1/3} \right\} \approx_\tau \tau^{-k\delta} \mups(B(z,\tau^{k/2})).\]
for  sufficiently large $k \in \mathbb{N}$.  This proves the  $\gtrsim_\tau$ direction of the lemma since 
\[
2\tau^{k+1}  \leq \tau^{k+2/3} \leq  \tau^{k+1/3} \leq \tau^k/2,
\]
recalling that $\tau = \tau_0^3 \leq 1/8$.
\end{proof}
An equivalent  definition of geometric finiteness due to Bowditch \cite[Definition (GF2)]{bowditch} guarantees the existence of $\lambda>0$ such that $\iota(\lset) \subseteq V_\lambda \cup \{ \infty\}$ where $V_\lambda$ is the Euclidean $\lambda$-neighbourhood of the span of $\{f_1 , \dots ,f_{\mathbf{k}(\mathbf{0})} \}$.  Fix $\tau$ as in Lemma \ref{tauball}.  Observe that $\iota^{-1}(V_\lambda)$ is the complement of two spheres tangent at $\mathbf{0}$ both of diameter $ \approx 1$.  Therefore,  noting that $\tau^k \leq |z|^2 \leq R < \tau^{k/2}$ by assumption,  $\iota^{-1}(V_\lambda)\cap B(z,\tau^{k/2}) $ is contained in a $\lesssim (|z| + \tau^{k/2})^2 \lesssim R$ neighbourhood of a $\mathbf{k}(\mathbf{0})$-dimensional plane, and so
\[
N_R\left(\lset \cap B(z,(2/\sqrt{C})\tau^{k/2}) \right) \lesssim N_R\left(\iota^{-1}(V_\lambda) \cap B(z,\tau^{k/2}) \right)  \lesssim \left(\frac{\tau^{k/2}}{R}\right)^{\mathbf{k}(\mathbf{0})}.
\] 
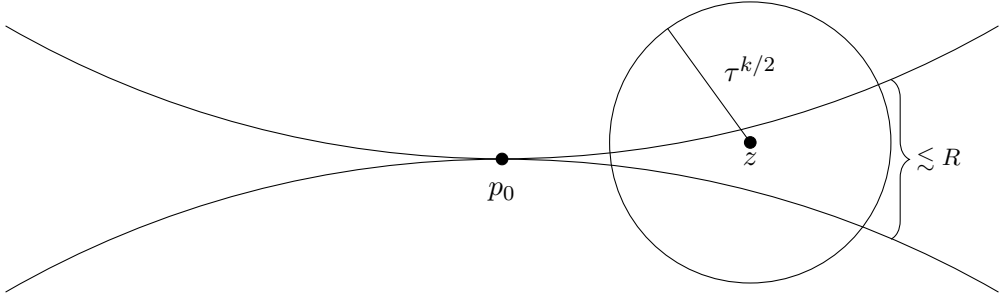
\begin{figure}[H]
\centering
\begin{tikzpicture}[scale=1.1]
\draw (0,1.61) arc (240:300:12cm);
\draw (0,-1.61) arc (120:60:12cm);
\node at (6,-0.4) {$p_0$};
\draw (9,0.2) -- (8,1.58);
\node at (9,1.1) {$\tau^{k/2}$};
\filldraw (9,0.2) circle (2pt) node[below] {$z$};
\draw (9,0.2) circle (1.7cm);
\filldraw (6,0) circle (2pt);
\draw [decorate,decoration={brace,amplitude=7pt},xshift=0pt,yshift=0pt]
(10.7,0.96) -- (10.7,-0.97)node [black,midway,xshift=18pt] {\footnotesize
$\lesssim R$};
\end{tikzpicture}
    \caption{An illustration showing how $\iota^{-1}(V_\lambda) \cap B(z,\tau^{k/2})$ is squeezed between two spheres tangent at $p_0$ (in this picture we have $\mathbf{k}(p_0)=1$).}
\end{figure}

Using this estimate and Lemma \ref{tauball}, apply the pigeonhole principle to find $x \in  B(z,(2/\sqrt{C})\tau^{k/2})$ such that
\[\# \left\{p \in P \cap B(x,R/3) \mid 2\tau^{k+1} \leq \size{H_p} < \tau^k/2 \right\} \gtrsim_\tau \tau^{-k\delta} \mups(B(z,\tau^{k/2})) \left(\frac{R}{\tau^{k/2}}\right)^{\mathbf{k}(\mathbf{0})}.\]
The idea  is to `pull' the horoballs in the above expression into our target set $B(z,R)$ using ($\iota$ conjugates of) $\{f_1 , \dots ,f_{\mathbf{k}(\mathbf{0})} \}$.  
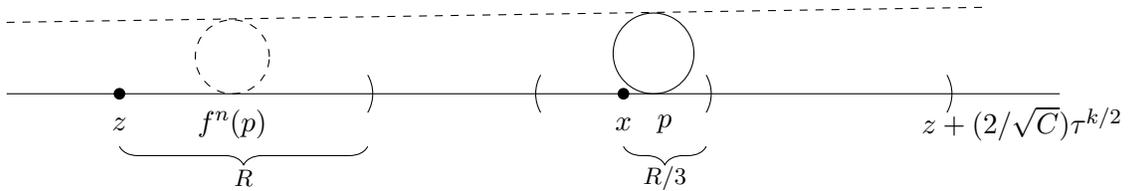
\begin{figure}[H]
    \centering
    \label{drag}
\begin{tikzpicture}
\draw (0,0) -- (14,0);
\draw (12.5,-0.25) arc (-30:30:0.5cm);
\filldraw (1.5,0) circle (2pt);
\filldraw (8.2,0) circle (2pt);
\node at (1.5,-0.4) {$z$};
\node at (8.75,-0.4) {$p$};
\node at (3,-0.4) {$f^n(p)$};
\node at (8.2,-0.4) {$x$};
\node at (13.5,-0.4) {$z+(2/\sqrt{C})\tau^{k/2}$};
\draw (9.3,-0.25) arc (-30:30:0.5cm);
\draw (7.1,-0.25) arc (210:150:0.5cm);
\draw (4.8,-0.25) arc (-30:30:0.5cm);
\draw (8.6, 0.537) circle (0.537cm);
\draw[dashed] (3, 0.492) circle (0.492cm);
\draw[dashed] (0,0.95) -- (13,1.15);
\draw [decorate,decoration={brace,mirror,amplitude=7pt},xshift=0pt,yshift=0pt]
(1.5,-0.7) -- (4.8,-0.7)node [black,midway,yshift=-12pt] {\footnotesize
$R$};
\draw [decorate,decoration={brace,mirror,amplitude=7pt},xshift=0pt,yshift=0pt]
(8.2,-0.7) -- (9.3,-0.7)node [black,midway,yshift=-12pt] {\footnotesize
$R/3$};
\end{tikzpicture}
    \caption{An illustration in the case where $d=1$ showing how the horoballs in the set $B(x,R/3)$ are moved into the set $B(z,R)$ by using a parabolic map $f$ fixing 0. The dashed arc represents the circle that the horoball $H_p$ is dragged along by repeated applications of $f$.}
    
\end{figure}

Let $\Gamma_0 = \langle f_1 , \dots ,f_{\mathbf{k}(\mathbf{0})} \rangle$ and note that  $\Gamma_0(\iota(x))$ is a lattice with `separation' 
\begin{equation} \label{separation}
\max_i |t_i| \approx |H_{\mathbf{0}}|^{-1}.
\end{equation}
 Since 
\[
\frac{|z|}{(2/\sqrt{C})\tau^{k/2}} \geq \frac{c_1 \tau^{k/2}}{(2/\sqrt{C})\tau^{k/2}}  = \frac{c_1\sqrt{C}}{2} 
\]
and
\[
\frac{|z|}{R} \geq \frac{c_1 \tau^{k/2}}{\tau^{k/2}} = c_1
\]
we can choose $c_1=c_1(\tau)$ sufficiently large to ensure that 
\begin{equation} \label{jacobian}
\frac{0.999}{|z|^2} \leq  \frac{1}{|y|^2}   \leq \frac{1.001}{|z|^2}
\end{equation}
for all $y \in B(z,(2/\sqrt{C})\tau^{k/2}+R) $. Therefore, we can choose $c_2$ sufficiently small to ensure
\begin{equation} \label{choose1}
|\iota(B(z,R))| \geq 2 |\iota(B(x,R/3))| \geq  \frac{R}{|z|^2} \geq  \frac{1}{c_2^2\size{H_{\mathbf{0}}}} \geq 100 \max\{\lambda,  \max_i |t_i|\}.
\end{equation}
Note that $c_2$ does not depend on $\size{H_{\mathbf{0}}}$ but does depend on the implicit constants in \eqref{separation}.  To deduce \eqref{choose1} we use that  $\iota$ is conformal and its  Jacobian derivative at $y \neq \mathbf{0}$ is $  1/|y|^2$  multiplied by an orthogonal matrix. Using \eqref{choose1} we immediately find $f \in \Gamma_0$ such that
\[
f(\iota(P \cap B(x,R/3))) \subseteq \iota(B(z,R))
\]
and therefore
\[
(\iota^{-1}f\iota)(P \cap B(x,R/3))) \subseteq B(z,R).
\]
Here the fact that $|\iota(B(z,R))|  \geq 100  \lambda$ was used to ensure that the orbit $\Gamma_0(\iota(x))$  cannot miss the target $\iota(B(z,R))$ as it passes noting that  $\Gamma_0(\iota(x))$ is contained in a $\mathbf{k}(\mathbf{0})$-dimensional plane which is itself a subset of $V_\lambda$.

For the above choice of  $f$, we also get
\begin{equation} \label{choose2}
1/2 \leq \frac{|\iota^{-1}f \iota(H_p)|}{ \size{H_p}} \leq 2.
\end{equation}
 To deduce  \eqref{choose2} we again use conformality of $\iota$ and \eqref{jacobian}.    Using \eqref{choose2}, for all $p \in P \cap B(x,R)$ satisfying $ 2\tau^{k+1} \leq \size{H_p} < \tau^k /2$, we get
\[
\tau^{k+1} \leq |\iota^{-1}f \iota(H_p)| < \tau^k.
\]
This completes the proof.  Finally, note that if   $\delta=\kmin=\kmax$, then $\mups$ is $\delta$-Ahlfors-David regular and the lower bound provided above agrees with the upper bound from Theorem \ref{theoretical} up to constants.

\subsection{Proof of Theorem \ref{proximity}} \label{proofproximity}
Clearly we can have $p_0 \in B(z,R)$ and $|H_{p_0}| \geq cR^{\lambda}$, so it remains to prove that if such $p_0$ is present, then it is unique. If $p_0$ is not unique, then we must be able to  associate an axes oriented right-angled triangle with   hypotenuse  of length $|H_{p_0}|/2$,   horizontal  side of length $|z-p_0|+R$ and opposite side strictly smaller than $|H_{p_0}|/2-cR^\lambda$ which has    vertices at the centre of $H_{p_0}$, the boundary of $H_{p_0}$, and (the vertex associated with the right-angle) above $p_0$ at `height' $> cR^\lambda$, see Figure \ref{pythagfig}.  Otherwise, using disjointness of horoballs, no other horoballs of size $\geq cR^\lambda$ can be tangent to a point in $B(z,R)$.
\begin{figure}[H]  
\centering
\begin{tikzpicture}[scale=1.05]
\draw (0,0) -- (10,0);
\draw (5,3) -- (7,0.77);
\draw (7,0.77) -- (5,0.77);
\draw (5,0.77) -- (5,3);
\draw[dashed] (7,0.77) -- (7,0);
\draw[dashed] (5,0.77) -- (5,0) node[below] {$p_0$};
\draw (5,3) circle (3cm);
\coordinate (A) at (5.2,0.77) {};
\coordinate (B) at (5.2,0.97) {};
\coordinate (C) at (5,0.97) {};
\draw (A) -- (B);
\draw (B) -- (C);
\draw [decorate,decoration={brace,amplitude=7pt},xshift=0pt,yshift=0pt]
(7,-0.5) -- (5,-0.5)node [black,midway,yshift=-12pt] {\footnotesize
$|z-p_0|+R$};
\draw [decorate,decoration={brace,amplitude=7pt},xshift=5pt,yshift=0pt]
(4.7,0.77) -- (4.7,3)node [black,midway,xshift=-42pt] {\footnotesize
$<|H_{p_0}|/2-cR^{\lambda}$};
\end{tikzpicture}
\caption{A picture of the associated  right-angled triangle.}

\label{pythagfig}
\end{figure}
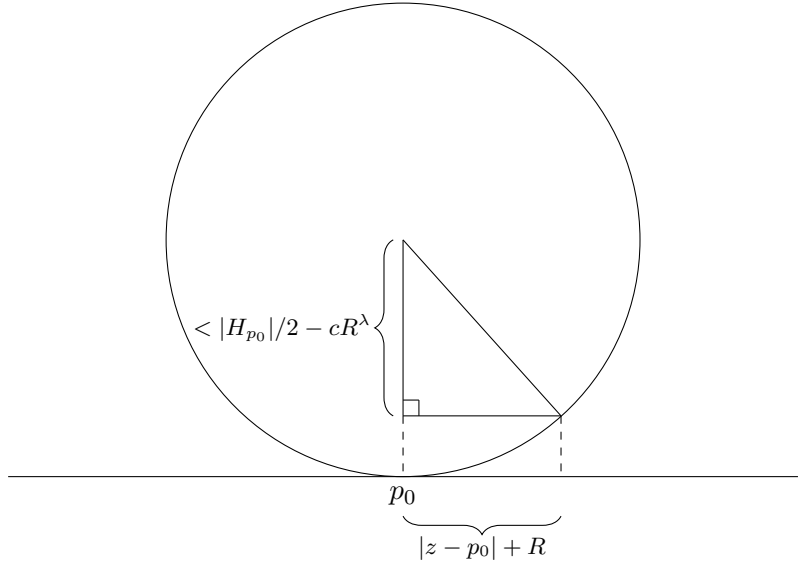

However, for such a triangle Pythagoras' Theorem ensures
\[
(|H_{p_0}|/2)^2 < (|H_{p_0}|/2-cR^\lambda)^2+(|z-p_0|+R)^2
\]
which in turn gives
\[
c|H_{p_0}|R^\lambda < c^2R^{2\lambda}+(|z-p_0|+R)^2
\]
contradicting  the assumption in the theorem.

\subsection{Proof of Theorem \ref{dimensions}} \label{proofdimensions}

We first estimate  $\ubox \mu_s$ from above.  By Theorem \ref{Global2} and  \eqref{GlobThm}, for all   $z \in \lset$ and $R>0$ sufficiently small,  
 \begin{align*}
\mu_s(B(z,R))  &\gtrsim   R^{2(s-\delta)} \mups(B(z,R)) \gtrsim R^{2(s-\delta)} R^{\max\{\delta,2\delta-\kmin\}} = R^{\max\{2s-\delta, 2s-\kmin\}} 
\end{align*}
which proves $\ubox \mu_s \leq \max\{2s-\delta, 2s-\kmin\}$.

We now estimate $\lbox \mu_s$ from below, which we do in two different ways to account for the maximum.  Let $p \in P$ be such that $\mathbf{k}(p) = \kmin$, and let $z_n \in \lset$ be such that $z_n \to p$ with $|z_n-p|$ strictly decreasing and $ |z_n-p|/|z_{n+1}-p| \to 1$. Such a sequence exists taking $z_n = f^n(z_0)$ for some $z_0 \neq p$ and $f$ a  parabolic map fixing $p$.  Choose a sequence of $R_n \to 0$ such that 
\[R_n < \size{z_n-p} < \frac{101R_n}{100}.\]
Note this means that for $n$ large enough, we have $\mathbf{k}(z_n, -\log R_n) = \kmin$ and $\rho(z_n, -\log R_n) \geq -\log R_n - C$ for some constant $C>0$, and so applying (\ref{GlobThm}) gives $\mups(B(z_n, R_n)) \lesssim R_n^{2\delta-\kmin}$. Furthermore, for $n$ large enough, there can only be finitely many parabolic points $p' \in B(z_n, R_n)$ satisfying   $  \size{H_{p'}}> 10 R_n^2 $, and they all must satisfy $\size{H_{p'}} \lesssim R_n^2$ by Theorem \ref{proximity}.  For a given $R \in (0, R_1)$, let $n$ be such that   $R_{n+1} \leq R < R_n$. Applying Theorem \ref{Global2} gives, for sufficiently large $n$,
\[\mu_s(B(z_n,R )) \leq \mu_s(B(z_n,R_n)) \lesssim R_n^{2(s-\delta)} R_n^{ 2\delta-\kmin} + R_n^{2s} = R_n^{2s-\kmin} \lesssim R^{2s-\kmin}   \]
which proves $\lbox \mu_s \geq 2s-\kmin$.

To derive the other lower bound, choose a sequence $p_n \in P$ such that $ \tau^{n+1} \leq \size{H_{p_n}}< \tau^n$ for some $\tau\in (0,1)$. Such a sequence (and $\tau$) exists by Theorem \ref{localvelani}.  For each $n \in \mathbb{N}$,   let $R_n>0$ denote the smallest real number such that $\rho(p_n, -\log R_n) = 0$ (note that such a number must exist as $p_n$ is parabolic) and therefore  $\size{H_{p_n}} \approx R_n$.   For a given $R \in (0, R_1)$, let $n$ be defined uniquely such that   $R_{n+1} \leq R < R_n$.  Then, by Theorem \ref{Global2},
\[\mu_s(B(p_n,  R)) \leq  \mu_s(B(p_n,  R_n)) \lesssim   R_n^{2s-\delta} + R_n^{s} + |H_{p_n}|^s \lesssim R_n^s \lesssim R^s\]
which proves $\lbox \mu_s \geq s$, as required.

As for $\asospec \mu_s$, the upper bound (for all $\theta \in (0,1)$) is a consequence of a general upper bound in terms of the upper box dimension proved in \cite[Proposition 4.1]{minkowski}.  For the lower bound, we first prove an estimate which holds for all $\theta \in (0,1)$.  Choose $p \in P$ such that $\mathbf{k}(p) = \kmin$ and $\size{H_p} \approx 1$ and choose $z_n \in \lset$ such that $z_n \to p$, and $R_n \to 0$ satisfying
\[R_n < \size{z_n-p} < \min\left\{R_n^\theta, \frac{101R_n}{100}\right\}.\]
Then, for sufficiently large $n$, 
\[\frac{\mu_s(B(z_n,R_n^\theta))}{\mu_s(B(z_n,R_n))} \gtrsim \frac{1}{R_n^{2s-\kmin}} = \left(\frac{R_n^\theta}{R_n}\right)^{(2s-\kmin)/(1-\theta)}\]
which gives $\asospec \mu_s \geq (2s-\kmin)/(1-\theta)$, as required.   This also proves $\qaso \mu_s = \infty$.

Next we prove a lower bound which will take on a different form depending on whether $\theta \leq 1/2$ or $\theta>1/2$.  Let $p, p'$ be distinct parabolic fixed points with $|H_p| \approx |H_{p'}| \approx 1$ and let $f$ be a parabolic element fixing $p'$. Let $p_n = f^n(p)$.  For large integers $n$  it is readily seen, e.g. \cite[Lemma 4.3]{fraserstuart}, that 
\[
|H_{p_n}|  =  |H_{f^n(p)}| = |f^n(H_{p})|  \approx |f^n(p) - p'|^2 \to 0
\]
  as $n \to 0$ with implicit constants depending on $f$ and $p$. First suppose $\theta < 1/2$ with the $\theta=1/2$ case following by continuity of the Assouad spectrum, see \cite{fraserbook}.  Choose $R_n = |H_{p_n}|$.  By Theorem \ref{Global2}, 
\[\mu_s(B(p_n, R_n)) \lesssim   R_n^{2s-\delta} + R_n^{s} + |H_{p_n}|^s \lesssim R_n^s.\]
Moreover, 
\[
|p_n - p'| =   |f^n(p) - p'| \lesssim   \sqrt{R_n} 
\]
Therefore,  since $\theta < 1/2$, for large enough $n$, $|p_n - p'| <  R_n^\theta$ and therefore
\[\frac{\mu_s(B(p_n, R_n^\theta))}{\mu_s(B(p_n, R_n))} \gtrsim \frac{|H_{p'}|^s}{R_n^{s}} \gtrsim \left(\frac{R_n^\theta}{R_n}\right)^{s/(1-\theta)}\]
which gives $\asospec \mu_s \geq s/(1-\theta)$, as required. Now suppose  $\theta>1/2$. We may assume $\delta<\kmin$ since otherwise the claimed bound does not improve on the previously established lower bound.   Choose   $R_n$ such that
\[
R_n^\theta = |p_n - p'|  \approx \sqrt{ |H_{p_n}| }.
\]
For large enough $n$, all horoballs tangent to a point in $B(p_n,R_n)$ have diameter $\lesssim  |H_{p_n}| \approx R_n^{2\theta}$.    Moreover, note that
\[
\exp(-\rho(p_n, -\log R_n) ) \gtrsim R_n^{2\theta-1}
\]
since $\rho(p_n, -\log R_n)$ can be bounded naively above by
\[
\log \left( \frac{R_n}{|H_{p_n}|}\right) \leq \log (R_n^{1-2\theta})
\]
up to an additive constant.  

\begin{figure}[H]
\centering
\begin{tikzpicture}
\draw (0,0) -- (10,0) node[right] {$\mathbb{S}^d$};
\draw (5,1.5) circle (1.5cm);
\draw (5,6) -- (5,0) node[below] {$p_n$};
\draw (0.8,6) arc (232:257:23cm);
\node at (8,3.5) {$H_{p'}$};
\filldraw (5,5) circle (2pt) node[right] {$(p_n)_{-\log R_n}$};
\filldraw (5,3) circle (2pt);
\draw [decorate,decoration={brace,amplitude=12pt},xshift=0pt,yshift=0pt]
(5,3) -- (5,5);
\node at (5.7,1.5) {$H_{p_n}$};
\draw (4.8,0) -- (4.8,0.2);
\draw (4.8,0.2) -- (5,0.2);
\end{tikzpicture}
\caption{Estimating $\rho(p_n, -\log R_n)$ from above by the hyperbolic distance between 
$(p_n)_{-\log R_n}$ and the `tip' of $H_{p_n}$.}
\end{figure}
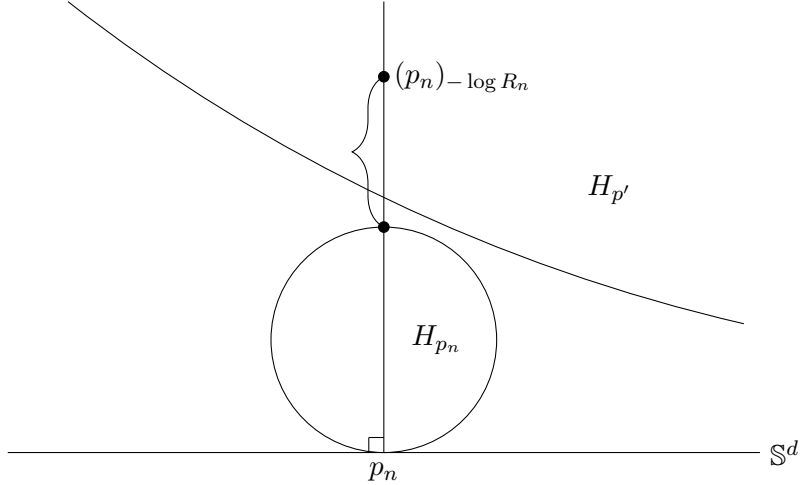

Therefore, by Theorem \ref{Global2} and \eqref{GlobThm} and using that $\delta<\kmin$,
\[\mu_s(B(p_n, R_n)) \lesssim   R_n^{2\theta(s-\delta)} \mu_\delta(B(p_n, R_n)) + |H_{p_n}|^s \lesssim  R_n^{2\theta(s-\delta)}R_n^{\delta+(2\theta-1)(\delta-\kmin)}.\]
 Therefore, for large enough $n$,
\[\frac{\mu_s(B(p_n, R_n^\theta))}{\mu_s(B(p_n, R_n))} \gtrsim \frac{|H_{p'}|^s}{R_n^{2\theta(s-\delta)+\delta+(2\theta-1)(\delta-\kmin)}} \gtrsim \left(\frac{R_n^\theta}{R_n}\right)^{(2\theta (s-\kmin)+\kmin)/(1-\theta)}\]
which gives $\asospec \mu_s \geq (2\theta (s-\kmin)+\kmin)/(1-\theta)$, as required.

\section*{Acknowledgements}
JMF was  supported by an \emph{EPSRC Standard Grant} (EP/R015104/1),  a \emph{Leverhulme Trust Research Project Grant} (RPG-2019-034), and an \emph{RSE Sabbatical Research Grant} (70249). LS was supported by the University of St Andrews.

\bibliographystyle{apalike}
\addcontentsline{toc}{section}{References}
\bibliography{ref}
\end{document}